\documentclass[11pt]{amsart}

\usepackage{amsthm,amsmath,amsfonts,amssymb,shuffle}
\usepackage{graphicx}
\usepackage{enumitem}
\usepackage{color}
\newtheorem{theorem}{Theorem}[section]
\newtheorem{lemma}[theorem]{Lemma}

\newtheorem{corollary}[theorem]{Corollary}

\theoremstyle{definition}
\newtheorem{definition}[theorem]{Definition}
\theoremstyle{remark}
\newtheorem{remark}[theorem]{Remark}

\numberwithin{equation}{section}

\newcommand{\des}{\ensuremath\mathrm{des}}

\newcommand{\R}{\ensuremath{R}}

\newcommand{\schubert}{\ensuremath\mathfrak{S}}
\newcommand{\fund}{\ensuremath\mathfrak{F}}
\newcommand{\stanley}{\ensuremath{S}}

\newlength\cellsize 

\savebox2{%
\begin{picture}(12,12)
\put(0,0){\line(1,0){12}}
\put(0,0){\line(0,1){12}}
\put(12,0){\line(0,1){12}}
\put(0,12){\line(1,0){12}}
\end{picture}}

\newcommand\cellify[1]{\def\thearg{#1}\def\nothing{}%
\ifx\thearg\nothing\vrule width0pt height\cellsize depth0pt%
  \else\hbox to 0pt{\usebox2\hss}\fi%
  \vbox to 12\unitlength{\vss\hbox to 12\unitlength{\hss$#1$\hss}\vss}}

\newcommand\tableau[1]{\vtop{\let\\=\cr
\setlength\baselineskip{-12000pt}
\setlength\lineskiplimit{12000pt}
\setlength\lineskip{0pt}
\halign{&\cellify{##}\cr#1\crcr}}}

\newcommand\nocellify[1]{\def\thearg{#1}\def\nothing{}%
\ifx\thearg\nothing\vrule width0pt height\cellsize depth0pt%
  \else\hbox to 0pt{\hss}\fi%
  \vbox to 12\unitlength{\vss\hbox to 12\unitlength{\hss$#1$\hss}\vss}}

\newcommand\notableau[1]{\vtop{\let\\=\cr
\setlength\baselineskip{-12000pt}
\setlength\lineskiplimit{12000pt}
\setlength\lineskip{0pt}
\halign{&\nocellify{##}\cr#1\crcr}}}

\savebox3{
  \scalebox{0.5}{%
\begin{picture}(24,24)
\put(0,12){\line(1,1){12}}
\put(0,12){\line(1,-1){12}}
\put(24,12){\line(-1,1){12}}
\put(24,12){\line(-1,-1){12}}
\end{picture}}}

\newcommand{\dia}[1]{\def\thearg{#1}\def\nothing{}%
\ifx\thearg\nothing\vrule width0pt height\cellsize depth0pt%
  \else\hbox to 0pt{\usebox3\hss}\fi%
  \vbox to 12\unitlength{\vss\hbox to 12\unitlength{\hss$#1$\hss}\vss}}

\savebox4{
\begin{picture}(12,12)
\put(6,6){\circle{12}}
\end{picture}}

\newcommand{\cir}[1]{\def\thearg{#1}\def\nothing{}%
\ifx\thearg\nothing\vrule width0pt height\cellsize depth0pt%
  \else\hbox to 0pt{\usebox4\hss}\fi%
  \vbox to 12\unitlength{\vss\hbox to 12\unitlength{\hss$#1$\hss}\vss}}

\begin{document}


\title[Schubert polynomials times Stanley polynomials]{Multiplication of a Schubert polynomial \\ by a Stanley symmetric polynomial}  

\author{Sami Assaf}
\address{Department of Mathematics, University of Southern California, Los Angeles, CA 90089}
\email{shassaf@usc.edu}

\subjclass[2010]{%
  Primary 14N15; %
  Secondary 05A05, 05A15, 14N10}




\keywords{Schubert polynomials, reduced expressions, Stanley symmetric polynomials}

\begin{abstract}
  We prove, combinatorially, that the product of a Schubert polynomial by a Stanley symmetric polynomial is a truncated Schubert polynomial. Using Monk's rule, we derive a nonnegative combinatorial formula for the Schubert polynomial expansion of a truncated Schubert polynomial. Combining these results, we give a nonnegative combinatorial rule for the product of a Schubert and a Schur polynomial in the Schubert basis.
\end{abstract}

\maketitle

%
\section{Schubert Calculus}
%
\label{sec:introduction}

Schubert calculus began around 1879 with Herman Schubert asking, and in special cases answering, enumerative questions in geometry\cite{Sch79}. For example, how many lines in space meet four given lines? To answer this, Schubert considered the case where the first line intersects the second and the third intersects the fourth, in which case the answer is two: the line connecting the two points of intersection and the line of intersection of the two planes spanned by the two pairs of intersecting lines. He then asserted, by his \emph{principle of conservation of number}, that the general answer, if finite, must also be two. Cohomology theory made rigorous Schubert's principle of conservation of number and lead us to modern Schubert calculus and intersection theory, which has ramifications in geometry, topology, combinatorics, and even plays a central role in string theory.

Lascoux and Sch{\"u}tzenberger \cite{LS82} defined polynomial representatives for the Schubert classes in the cohomology ring of the complete flag variety. These \emph{Schubert polynomials} give explicit polynomial representations of the Schubert classes so that intersections numbers can be read off the structure constants. That is, the structure constants for Schubert polynomials, $c_{u,v}^{w}$, defined by
\begin{displaymath}
  \schubert_{u} \cdot \schubert_{v} = \sum_{w} c_{u,v}^{w} \schubert_{w},
\end{displaymath}
enumerate flags in a suitable triple intersection of Schubert varieties. Therefore these so-called \emph{Littlewood--Richardson coefficients} are known to be nonnegative. A fundamental open problem in Schubert calculus is to find a \emph{positive} combinatorial construction for $c_{u,v}^{w}$.

In the special case of the Grassmannian subvariety, Schubert polynomials are Schur polynomials \cite{LS82}, and there are many combinatorial rules for computing $c_{u,v}^{w}$ when $u$ and $v$ are both grassmannian with the same number of variables \cite{LR34}. Beyond this, Sottile \cite{Sot96} proved a Pieri formula for computing computing $c_{u,v}^{w}$ when $v$ is a grassmannian permutation whose corresponding partition has one row or one column. In this paper, we give a combinatorial rule for the Schubert expansion of the product of a Schubert polynomial and a Schur polynomial, thus computing $c_{u,v}^{w}$ when $v$ is grassmannian in at least as many variables as $u$. Our result is a special case of a more general rule for multiplying a Schubert polynomial by a Stanley symmetric polynomial. 

Our first main result may be stated as
\begin{displaymath}
  \schubert_{u \times v}(x_1,\ldots,x_k) = \schubert_u \stanley_{v}(x_1,\ldots,x_k),
\end{displaymath}
where $\stanley_w$ is the Stanley symmetric polynomial associated to $w$ \cite{Sta84}. This motivates understanding truncations of Schubert polynomials which we do by giving the following formula for the truncation of a Schubert polynomial by setting the last variable to $0$,
\begin{displaymath}
  \schubert_w(x_1,\ldots,x_{k-1},0) = \sum_{\substack{u = \hat{w} (a_1,k) \cdots (a_m,k+m-1) \\ a_i < k \hspace{1em} \mathrm{and} \hspace{1em} \ell(u) = \ell(w)}} \schubert_u,
\end{displaymath}
where $k$ is the final descent of $w$, $k+m = \max\{i \mid w_k > w_i\}$, and $\hat{w} = w_1 \cdots w_{k-1} w_{k+1} \cdots w_{k+m} w_k$. Finally, we combine these results to show that for $v(\lambda,k)$ the grassmannian permutation associated to $\lambda$ and $k$,where $k$ is at least the final descent position for $w$, $c^{w}_{u,v(\lambda,k)}$ enumerates certain chains $w = u (a_1,b_1) \cdots (a_{|\lambda|},b_{|\lambda|})$ where $a_i \leq k < b_i$. In particular, this gives a purely combinatorial proof that $c^{w}_{u,v(\lambda,k)}$ is a nonnegative integer.

%
\section{Schubert polynomials}
%
\label{sec:schubert}

A \emph{reduced expression} is a sequence $\rho = (i_k, \ldots, i_1)$ such that the permutation $s_{i_k} \cdots s_{i_1}$ has length $k$, where $s_i$ is the simple transposition that interchanges $i$ and $i+1$. Let $\R(w)$ denote the set of reduced expressions for $w$. For example, the elements of $\R(42153)$ are shown below:

  \begin{displaymath}
    \begin{array}{cccccc}
      (4,2,1,2,3) & (4,1,2,1,3) & (4,1,2,3,1) & (2,4,1,2,3) & (2,1,4,2,3) & (2,1,2,4,3) \\
      (1,4,2,3,1) & (1,2,4,3,1) & (1,4,2,1,3) & (1,2,4,1,3) & (1,2,1,4,3) & 
    \end{array}
  \end{displaymath}

For $\rho \in R(w)$, say that a strong composition $\alpha$ is \emph{$\rho$-compatible} if $\alpha$ is weakly increasing with $\alpha_j < \alpha_{j+1}$ whenever $\rho_j < \rho_{j+1}$ and $\alpha_j \leq \rho_j$. For example, there are two compatible sequences for $(4,2,1,2,3)$, namely $(1,1,1,2,4)$ and $(1,1,1,2,3)$, and there is one compatible sequence for $(2,4,1,2,3)$, namely $(1,1,1,2,2)$. None of the other reduced expressions for $42153$ has a compatible sequence.

\begin{definition}[\cite{BJS93}]
  The Schubert polynomial $\schubert_w$ is given by
  \begin{equation}
    \schubert_w = \sum_{\rho \in R(w)} \sum_{\alpha \ \rho-\mathrm{compatible}} x_{\alpha_1} \cdots x_{\alpha_{\ell(w)}},
  \end{equation}
  where the sum is over compatible sequences $\alpha$ for reduced expressions $\rho$.
  \label{def:schubert}
\end{definition}

For example, we can compute
\begin{displaymath}
  \schubert_{42153} = x_1^3 x_2 x_4 + x_1^3 x_2 x_3 + x_1^3 x_2^2.
\end{displaymath}

Let $1^m \times w$ denote the permutation obtained by adding $m$ to all values of $w$ in one-line notation and pre-pending $1,2,\ldots,m$. Note that the reduced expressions for $1^m \times w$ are simply those for $w$ with each index increased by $m$. To make the example more interesting, consider $1 \times 42153 = 153264$. Then seven reduced expressions contribute to the Schubert polynomial, giving
\begin{eqnarray*}
  \schubert_{153264} & = & x_1^3 x_2^2+2 x_1^3 x_2 x_3+x_1^3 x_2 x_4+x_1^3 x_2 x_5+x_1^3 x_3^2+x_1^3 x_3 x_4+x_1^3 x_3 x_5+x_1^2 x_2^3+2 x_1^2 x_2^2 x_3 \\
  & & + x_1^2 x_2^2 x_4+x_1^2 x_2^2 x_5+x_1^2 x_2 x_3^2+x_1^2 x_2 x_3 x_4+x_1^2 x_2 x_3 x_5+2 x_1 x_2^3 x_3+x_1 x_2^3 x_4 \\
  & & +x_1 x_2^3 x_5 + x_1 x_2^2 x_3^2+x_1 x_2^2 x_3 x_4+x_1 x_2^2 x_3 x_5+x_2^3 x_3^2+x_2^3 x_3 x_4+x_2^3 x_3 x_5 .
\end{eqnarray*}

We harness the power of the \emph{fundamental slide polynomials} of Assaf and Searles \cite{AS17} to give a condensed formula for Schubert polynomials. Given a weak composition $a$, let $\mathrm{flat}(a)$ denote the strong composition obtained by removing all zero parts.

\begin{definition}[\cite{AS17}]
  For a weak composition $a$ of length $n$, define the \emph{fundamental slide polynomial} $\fund_{a} = \fund_{a}(x_1,\ldots,x_n)$ by
  \begin{equation}
    \fund_{a} = \sum_{\substack{b \geq a \\ \mathrm{flat}(b) \ \mathrm{refines} \ \mathrm{flat}(a)}} x_1^{b_1} \cdots x_n^{b_n},
    \label{e:fund-shift}
  \end{equation}
  where $b \geq a$ means $b_1 + \cdots + b_k \geq a_1 + \cdots + a_k$ for all $k=1,\ldots,n$.
  \label{def:fund-shift}
\end{definition}

For example, we compute
\begin{eqnarray*}
  \fund_{(0,3,1,0,1)} & = & x_2^3 x_3 x_5 + x_2^3 x_3 x_4 + x_1 x_2^2 x_3 x_5 + x_1 x_2^2 x_3 x_4 + x_1^2x_2 x_3 x_5 \\
  & & + x_1^2x_2 x_3 x_4 + x_1^3 x_3 x_5 + x_1^3 x_3 x_4 + x_1^3 x_2 x_5 + x_1^3 x_2 x_4 + x_1^3 x_2 x_3.
\end{eqnarray*}

To facilitate virtual objects as defined below, we extend notation and set
\begin{equation}
  \fund_{\varnothing} = 0.
\end{equation}

The \emph{run decomposition} of a reduced expression $\rho$ partitions $\rho$ into increasing sequences of maximal length. We denote the run decomposition by $(\rho^{(k)} | \cdots | \rho^{(1)})$. For example, the run decomposition of $(5,6,3,4,5,7,3,1,4,2,3,6)$, a reduced expression for $41758236$, is $(5,6|3,4,5,7|3|1,4|2,3,6)$. 

\begin{definition}
  For a reduced expression $\rho$ with run decomposition $(\rho^{(k)} | \cdots | \rho^{(1)})$, set $r_k = \rho^{(k)}_1$ and, for $i<k$, set $r_i = \min(\rho^{(i)}_1,r_{i+1}-1)$. Define the \emph{weak descent composition of $\rho$}, denoted by $\des(\rho)$, by $\des(\rho)_{r_i} = |\rho^{(i)}|$ and all other parts are zero if all $r_i>0$ and $\des(\rho) = \varnothing$ otherwise.
  \label{def:des-red}
\end{definition}

We say that $\rho$ is \emph{virtual} if $\des(\rho) = \varnothing$. For example, $(5,6,3,4,5,7,3,1,4,2,3,6)$ is virtual since $r_1=0$. Let $0^m \times a$ denote the weak composition obtained by pre-pending $m$ zeros to $a$. Then for $\rho \in R(w)$ non-virtual, the corresponding reduced expression for $R(1^m \times w)$ will have weak descent composition $0^m \times \des(\rho)$. For example, the weak descent composition for $(6,7,4,5,6,8,4,2,5,3,4,7)$, a reduced expression for $152869347$, is $(3,2,1,4,0,2)$. Note the reversal from the run decomposition to the descent composition. To revisit the previous example, the non-virtual reduced expressions for $153264$ are given below:
\begin{displaymath}
  \begin{array}{ccccccc}
    (5,3,2,3,4) & (5,2,3,2,4) & (5,2,3,4,2) & (3,5,2,3,4) & (3,2,5,3,4) & (3,2,3,5,4) & (2,3,5,2,4) 
  \end{array}
\end{displaymath}

\begin{theorem}
  For $w$ any permutation, we have
  \begin{equation}
    \schubert_{w} = \sum_{\rho \in \R(w)} \fund_{\des(P)},
    \label{e:schubert-slide}
  \end{equation}
  where the sum may be taken over non-virtual reduced expressions $\rho$.
  \label{thm:schubert-slide}
\end{theorem}

\begin{proof}
  Map each compatible sequence $\alpha$ to the weak composition $a$ whose $i$th part is the number of $j$ such that $\alpha_j = i$. For example, the compatible sequence $(1,1,1,2,4)$ for the reduced expression $(4,2,1,2,3)$ maps to the weak composition $(3,1,0,1)$. The greedy choice of a compatible sequence takes each $\alpha_i$ as large as possible. Under the correspondence, this precisely becomes $\des(\rho)$ since the condition $\alpha_j < \alpha_{j+1}$ whenever $\rho_j < \rho_{j+1}$ corresponds precisely to taking $r_i < r_{i+1}$ and the conditions $\alpha_j \leq \rho_j$ and $\alpha_j \leq \alpha_{j+1}$ correspond precisely to $r_i \leq \rho_1^{(i)}$. Furthermore, $\des(\rho) = \varnothing$ precisely when $\rho$ admits no compatible sequences.

  Given a compatible sequence for $\rho$, we may decrement parts provided we maintain $\alpha_j < \alpha_{j+1}$ whenever $\rho_j < \rho_{j+1}$, and this corresponds precisely to sliding parts of the weak composition left, possibly breaking them into refined pieces. Every compatible sequence may be obtained from the greedy one in this way, just as every term in the monomial expansion of the fundamental slide polynomial arises in the analogous way. 
\end{proof}

For example, the two non-virtual reduced expressions for $42153$ give
\begin{displaymath}
  \schubert_{42153} = \fund_{(3,1,0,1)} + \fund_{(3,2,0,0)},
\end{displaymath}
a slight savings over the monomial expansion. Bumping this example up to $153264$, we have
\begin{displaymath}
  \schubert_{153264} = \fund_{(0,3,1,0,1)} + \fund_{(2,2,0,0,1)} + \fund_{(1,3,0,0,1)} + \fund_{(0,3,2,0,0)} + \fund_{(2,2,1,0,0)} + \fund_{(1,3,1,0,0)} + \fund_{(2,3,0,0,0)},
\end{displaymath}
which is considerably more compact than the $26$-term monomial expansion. Furthermore, this paradigm shift to fundamental slide generating functions facilitates the applications to follow.

%
\section{Stanley symmetric polynomials}
%
\label{sec:stanley}

In order to enumerate reduced expressions, Stanley introduced a new family of symmetric functions \cite{Sta84}. They may be defined by
\begin{equation}
   \stanley_w(x_1,x_2,\ldots)  = \sum_{\rho \in \R(w)} F_{\mathrm{Des}(\rho)}(x_1,x_2,\ldots),
  \label{e:stanley}
\end{equation}
where $\mathrm{Des}(\rho) = (|\rho^{(1)}|,\ldots,|\rho^{(k)}|)$ is the strong descent composition for $\rho$, and $F_{\alpha}$ denotes Gessel's \emph{fundamental quasisymmetric function} \cite{Ges84} defined by
\begin{equation}
  F_{\alpha}(x_1,x_2,\ldots) = \sum_{\mathrm{flat}(b) \ \mathrm{refines} \ \alpha} x^{b}.
  \label{e:gessel}
\end{equation}

For example, all eleven elements of $\R(42153)$ contribute to the Stanley function, and we have
\begin{displaymath}
  \stanley_{42153} = F_{(3,1,1)} + 2 F_{(2,2,1)} + 2 F_{(1,3,1)} + F_{(3,2)}  + 2 F_{(1,2,2)} + F_{(1,1,3)} + F_{(2,1,2)} + F_{(2,3)} .
\end{displaymath}

Note that $\stanley_w$ and $F_{\alpha}$ are defined as \emph{functions}, though we can make them polynomials in $k$ variables by setting $x_i=0$ for all $i>k$.

\begin{lemma}
  Let $a$ be a weak composition, say $(a_1,\ldots,a_{\ell})$. Suppose there exist $1 \leq r\leq s \leq \ell$ such that $a_i = 0$ for $i<r$ and $a_i \neq 0$ for $r \leq i \leq s$. Then for any $k \leq s$, we have 
  \begin{equation}
    \fund_{a} (x_1,\ldots,x_k) = F_{\mathrm{flat}(a)}(x_1,\ldots,x_k).
    \label{e:fund-trunc}
  \end{equation}
  \label{lem:fund-trunc}
\end{lemma}

\begin{proof}
  It is enough to show that for $b = (b_1,\ldots,b_k)$ such that $\mathrm{flat}(b)$ refines $\mathrm{flat}(a)$, we have $b \geq a$. If $\beta = (\beta_1,\ldots,\beta_n)$ refines $\alpha = (\alpha_1,\ldots,\alpha_m)$, then $n \geq m$ and for any $i < m$ we have
  \begin{displaymath}
    \beta_{n-i} + \cdots + \beta_{n} \leq \alpha_{m-i} + \cdots + \alpha_{m}.
  \end{displaymath}
  Therefore, since $k\leq s$, for any $r\leq i\leq k$ we have 
  \[ b_{i+1} + \cdots + b_{k} \leq a_{i+1} + \cdots + a_{\ell} \]
  since the left sum has at most $k-i$ nonzero terms and the right sum has at least $s-i$ nonzero terms. Since both compositions add to the same value, we must have
  \[ b_1 + \cdots + b_i \geq a_1 + \cdots + a_i. \]
  Finally, for any $i<r$, the right sum above is $0$, so the relation $b\geq a$ holds.
\end{proof}

Given permutations $u$ and $v$ and a positive integer $m$ such that $u_i = i$ for all $i>m$, define the permutation $u \times_m v$ by $(u \times_m v)_i = u_i$ for $i=1\ldots m$ and $(u \times_m v)_i = m+v_i$ for $i>m$. Comparing with our prior notation, we have $1^k\times w = 1 \times_k w$. For notational simplicity, we drop the subscript $m$ whenever it is taken to be minimal.

\begin{theorem}
  Let $u,v$ permutations with $m$ minimal such that $u_i=i$ for all $i>m$. For integers $k,n$ with $m \leq k \leq n$, we have
  \begin{equation}
    \schubert_{u \times_n v}(x_1,\ldots,x_k) = \schubert_u \stanley_{v} (x_1,\ldots,x_k).
    \label{e:cross}
  \end{equation}
  \label{thm:cross}
\end{theorem}

\begin{proof}
  For $u$ and $v$ arbitrary permutations with $m$ the minimal index for which $u_i = i$ for all $i>m$, the largest index $i$ for which $s_i$ occurs in an element of $\R(u)$ is at most $m-1$ and the smallest index $j$ for which $s_j$ occurs in an element of $\R(1^n \times v)$ is least $n+1 \geq m+1$. Therefore $s_i$ and $s_j$ commute for $s_i$ any term in an element of $\R(u)$ and $s_j$ any term in an element of $\R(1^n \times v)$. In particular, we have a simple bijection $\R(u \times_n v) \rightarrow \R(u) \shuffle \R(1^n \times v)$, where $U\shuffle V$ denotes the usual shuffle of words $U$ and $V$ that interleaves letters in all possible ways while maintaining the relative order of letters from $U$ and of letters from $V$.

  Assaf and Searles \cite{AS17} generalized the shuffle product of Eilenberg and Mac Lane \cite{EM53} to the \emph{slide product} for weak compositions and showed that
  \begin{equation}
    \fund_a \fund_b = \sum_{c} [c \mid  \in a \shuffle b] \fund_c,
  \end{equation}
  where $[c \mid  \in a \shuffle b]$ denotes the coefficient of $c$ in the slide product $a \shuffle b$ which is precisely the nonvirtual terms in the shuffle product. In particular, this shows that
  \begin{displaymath}
    \schubert_{u \times_n v} = \sum_{\rho \in \R(u \times_n v)} \fund_{\des(\rho)} = \sum_{(\sigma,\tau) \in \R(u) \times R(1^n \times v)} \fund_{\des(\sigma)} \fund_{\des(\tau)} = \schubert_u \schubert_{1^n \times v}.
  \end{displaymath}
  Therefore, to prove \eqref{e:cross}, it suffices to show that for $n\geq k$, we have
  \begin{equation}
    \schubert_{1^n \times w} (x_1,\ldots,x_k) = \stanley_w(x_1,\ldots,x_k).
    \label{e:1down}
  \end{equation}

  Let $\rho \in \R(w)$, and let $\hat{\rho} \in \R(1^n\times w)$ be such that $\hat{\rho}_i = \rho_i + n$. Recall from \S\ref{sec:schubert} that this correspondence gives a bijection between $\R(w)$ and $\R(1^n\times w)$. If $\rho\in\R(w)$ is nonvirtual, then $\des(\hat{\rho}) = 0^n\times\des(\rho)$, and so by Lemma~\ref{lem:fund-trunc}, we have
  \[ \fund_{\des(\hat{\rho})} (x_1,\ldots,x_k) = \fund_{0^n \times \des(\rho)}(x_1,\ldots,x_k) = F_{\des(\rho)}(x_1,\ldots,x_k).\]
  Suppose then that $\rho$ is virtual. With notation as in Definition~\ref{def:des-red}, $r_i<1$ only if $r_i = r_{i+1}-1$. Setting $R = 1-\min(1,r_1) \geq 0$, we have $\des(1^R \times \rho) \neq \varnothing$ with the first $R+1$ terms nonzero. When $n\geq R$, Lemma~\ref{lem:fund-trunc} applies with $r = n-R$ and $s = n+1$ and gives
  \[ \fund_{\des(\hat{\rho})} (x_1,\ldots,x_k) = \fund_{0^{n-R} \times \des(1^{R}\times\rho)} (x_1,\ldots,x_k) = F_{\mathrm{Des}(\rho)} (x_1,\ldots,x_k).\]
  If $n<R$, then $\hat{\rho}$ is virtual and the length of $\mathrm{Des}(\rho)$ is greater than $n$, giving
  \[ \fund_{\des(\hat{\rho})} (x_1,\ldots,x_k) = 0 = F_{\mathrm{Des}(\rho)} (x_1,\ldots,x_k).\]
  Combining these cases, we have
  \begin{displaymath}
    \schubert_{1^n \times w} (x_1,\ldots,x_k) = \sum_{\rho \in \R(1^n \times w)} \fund_{\des(\rho)} (x_1,\ldots,x_k) 
    = \sum_{\rho\in\R(w)} F_{\mathrm{Des}(\rho)}(x_1,\ldots,x_k) = \stanley_w(x_1,\ldots,x_k),
  \end{displaymath}
  thereby completing the proof.
\end{proof}

Theorem~\ref{thm:cross} motivates developing an understanding of truncating Schubert polynomials. To make this more compelling, we have the following corollary for the case when $v$ is \emph{grassmannian}, that is, when $v$ has at most one descent. 

Given a partition $\lambda$ of length $j$ and a positive integer $k \geq j$, the \emph{grassmannian permutation associated to $\lambda$ and $k$}, denoted by $v(\lambda,k)$, is given by
\begin{equation}
  v(\lambda,k)_i = i + \lambda_{k-i+1}
  \label{e:grass}
\end{equation}
for $i = 1,\ldots,k$, where we take $\lambda_i=0$ for $i>j$, and $v(\lambda,k)$ has a unique descent at $k$. For example,
\begin{displaymath}
  \begin{array}{rcrrrrrrrrrrr}
    & & 0 & 0 & 1 & 4 & 4 & 5 \ \vline & & & & & \\\cline{3-13}
    v((5,4,4,1), 6) & = & 1 & 2 & 4 & 8 & 9 & 1\!1 \ \vline & 3 & 5 & 6 & 7 & 1\!0 . 
  \end{array}
\end{displaymath}
It is easy to see that $v(\lambda,k)$ gives a bijection between grassmannian permutations with unique descent at $k$ and partitions of length at most $k$.

Abusing history, define the \emph{Schur polynomial for $\lambda$ in $k$ variables}, denoted by $s_{\lambda}(x_1,\ldots,x_k)$, by the following result of Lascoux and Sch{\"u}tzenberger \cite{LS82}
\begin{equation}
  s_{\lambda}(x_1,\ldots,x_k) = \schubert_{v(\lambda,k)}.
  \label{e:schur}
\end{equation}

\begin{corollary}
  For $v = v(\lambda,k)$ a grassmannian permutation and $u$ any permutation with last descent at or before $k$, and for any $\ell\geq k$ we have
  \begin{equation}
    \schubert_{u} \schubert_{v} = \schubert_{u} s_{\lambda}(x_1,\ldots,x_k) = \schubert_{u \times_{\ell} v}(x_1,\ldots,x_k).
    \label{e:kohnert}
  \end{equation}
  \label{cor:kohnert}
\end{corollary}

\begin{proof}
  For $a$ a weak composition of length $k$, Assaf and Searles \cite{AS17} showed that
  \begin{displaymath}
    \fund_{a} = F_{\mathrm{flat}(a)} (x_1,\ldots,x_k)
  \end{displaymath}
  if and only if there exists an index $1\leq j\leq k$ such that $a_i=0$ if and only if $i<j$. Moreover, they showed that $\schubert_w$ is upper-unitriangular with respect to the fundamental basis. Therefore we conclude that, for $w$ any permutation and $k$ a positive integer, we have
  \begin{displaymath}
    \stanley_{w} (x_1,\ldots,x_k) = \schubert_w
  \end{displaymath}
  if and only if $w$ is grassmannian with descent at or before $k$. The result follows from \eqref{e:schur}.
\end{proof}

\begin{remark}
  Corollary~\ref{cor:kohnert} was first asserted by Kohnert \cite{Koh91}, though it lacked rigorous proof.
\end{remark}

%
\section{Schubert expansions}
%
\label{sec:trunc}

Since Schubert polynomials are polynomial representatives for Schubert classes in the cohomology ring \cite{LS82}, the structure constants for Schubert polynomials, $c_{u,v}^{w}$, defined by
\begin{equation}
  \schubert_{u} \cdot \schubert_{v} = \sum_{w} c_{u,v}^{w} \schubert_{w},
\label{e:basis}
\end{equation}
enumerate flags in a suitable triple intersection of Schubert varieties. A fundamental open problem in Schubert calculus is to find a \emph{positive} combinatorial construction for $c_{u,v}^{w}$. By Corollary~\ref{cor:kohnert}, we can solve this problem for $v$ grassmannian by giving a combinatorial formula for the truncation of a Schubert polynomial. The key to proving our truncation formula is the well-known Monk's rule \cite{Mon59}, which computes $c_{u,v}^{w}$ whenever $\ell(v)=1$.

\begin{lemma}[Monk's Rule]
  For $w$ a permutation and $m,k$ positive integers, we have
  \begin{equation}
    \schubert_w \cdot (x_1 + \cdots + x_k) = \sum_{\substack{a \leq k < b \\ \ell(w (a,b)) = \ell(w) + 1}} \schubert_{w (a,b)} .
    \label{e:monk}
  \end{equation}
  \label{lem:monk}
\end{lemma}

\begin{theorem}
  Let $w$ be a permutation with final descent at position $k$. Then
  \begin{equation}
    \schubert_w(x_1,\ldots,x_{k-1},0) = \sum_{\substack{u = \hat{w} (a_1,k) \cdots (a_m,k+m-1) \\ a_i < k \hspace{1em} \mathrm{and} \hspace{1em} \ell(u) = \ell(w)}} \schubert_u,
    \label{e:trunc}
  \end{equation}
  where $k+m = \max\{i \mid w_k > w_i\}$, and $\hat{w} = w_1 \cdots w_{k-1} w_{k+1} \cdots w_{m} w_k$.
  \label{thm:trunc}
\end{theorem}

\begin{proof}
  We begin with the simple observation that
  \[ x_k = (x_1 + \cdots + x_k) - (x_1 + \cdots + x_{k-1}). \]
  Multiplying through by $\schubert_v$, applying Monk's rule and canceling terms, we get
  \begin{displaymath}
    \schubert_{v} x_k = \sum_{\substack{a \leq k < b \\ \ell(v (a,b)) = \ell(v) + 1}} \hspace{-2ex} \schubert_{v (a,b)} 
    - \sum_{\substack{a < k \leq b \\ \ell(v (a,b)) = \ell(v) + 1}} \hspace{-2ex} \schubert_{v (a,b)} 
    = \sum_{\substack{k < b \\ \ell(v (k,b)) = \ell(v) + 1}} \hspace{-2ex} \schubert_{v (k,b)} 
    - \sum_{\substack{a < k \\ \ell(v (a,k)) = \ell(v) + 1}} \hspace{-2ex} \schubert_{v (a,k)} 
  \end{displaymath}
  If $v$ has no descent beyond position $k$, then there is a unique $b$ such that $\ell(v (k,b)) = \ell(v)+1$, namely $b = \max\{j>b \mid v_b < v_j\}$. Therefore, for this case, setting $x_k=0$ yields
  \begin{equation}
    \schubert_{v(k,b)}(x_1, \ldots, x_{k-1}) = \sum_{\substack{a < k \\ \ell(v (a,k)) = \ell(v) + 1}} \schubert_{v (a,k)} (x_1, \ldots, x_{k-1}).
    \label{e:bridge}
  \end{equation}
  Beginning with $w$ such that the final descent of $w$ is at position $k$, set $b = \max\{j \mid w_k > w_j\}$. Then taking $v = w(k,b)$, we consider all $u=v (a,k) = w (k,b) (a,k)$ where $a<k$ appears in the right hand side of \eqref{e:bridge}. Clearly $u$ has no descent beyond $k$. However, if $u$ has a descent at $k$, we may expand it similarly, noting that $\max\{j \mid u_k > u_j\} < b$. Thus the process ultimately yields permutation with no descent at or beyond $k$, and we have
  \begin{equation}
    \schubert_w(x_1,\ldots,x_{k-1}) = \sum_{\substack{(a_1,\ldots,a_n) < k < (b_1,\ldots,b_n) \\ \ell(w (k,b_1)(a_1,k) \cdots (k,b_n)(a_n,k) ) = \ell(w)}} \schubert_{w (k,b_1)(a_1,k) \cdots (k,b_n)(a_n,k)},
    \label{e:down-up}
  \end{equation}
  where $b_1>\cdots>b_n$ are determined by $w$ and the $a_i$'s, and the $a_i$'s are necessarily distinct. To derive \eqref{e:trunc} from \eqref{e:down-up}, we describe a simple bijection between the transition sequences appearing in each expansion. Note that each transposition $(k,b_i)$ decreases the length, and each transposition $(a_i,k)$ increases it; call these down and up transpositions, respectively.
  
  Beginning with a sequence $(k,b_1)(a_1,k) \cdots (k,b_n)(a_n,k)$, between any two pairs $(k,b_i)(a_i,k)$ and $(k,b_{i+1})(a_{i+1},k)$ insert terms $(k,j)(k,j)$ for $j=b_i-1,\ldots,b_{i+1}+1$, and regard the left $(k,j)$ as down and the right $(k,j)$ as up. This increases the number of transpositions to $m$, where $k+m = \max\{i \mid w_k > w_i\}$. Move every down transposition left using the commutativity relations
  \[ (a_{i-1}, k) (k, b_{i}) = (k, b_{i}) (a_{i-1}, b_{i}) \hspace{1em} \mbox{or} \hspace{1em} (k, b_{i-1}) (k, b_{i}) = (k, b_{i}) (b_{i-1}, b_{i}),  \]
  where here the left transposition is the first up transposition crossed. Note that each up transposition will be crossed at most once by a non-commuting down transposition. The result is
  \[ (k, b_1) \cdots (k, b_m) (a_1 b_2) \cdots (a_{m-1} b_m) (a_m k), \]
  where we have $b_i = (k+m)-i+1$, and we write any up transposition of the form $(b_{j-1} b_{j})$ with $a_j$ the smaller entry. Finally, we reverse the order of the up transpositions, noting that they commute except for one case, for which we use the commutativity relation
  \[ (a_{i-1}, b_i) (b_{i-1}, b_i) = (a_{i-1},b_{i-1}) (a_{i-1}b_{i}). \]
  This process is clearly reversible, thus establishing the desired bijection to complete the proof.
\end{proof}

For example, take $w=51738246$ and $k=5$. Then we have $\hat{w}=51732468$ with $b_1,b_2,b_3 = 5,6,7$, and so $a_1,a_2,a_3 = 2,4,4$ or $2,4,1$, giving
\begin{displaymath}
  \schubert_{51738246}(x_1,\ldots,x_4,0) = \schubert_{5276134} + \schubert_{6274135}.
\end{displaymath}
The corresponding example for the bijection (in reverse) given in the proof for $u=5276134$ is
\begin{displaymath}
  (5,8)(5,7)(5,6)(2,5)(4,6)(4,7) = (5,8)(4,5)(5,7)(5,7)(5,6)(2,5) = (5,8)(4,5)(5,6)(2,5).
\end{displaymath}

In particular, Theorem~\ref{thm:trunc} shows that a truncated Schubert polynomial is Schubert positive. Combining this with Theorem~\ref{thm:cross}, we have a combinatorial proof of the following.

\begin{corollary}
  The product of a Schubert polynomial in $j\leq k$ variables and a Stanley symmetric polynomial in $k$ variables expands nonnegatively in the Schubert basis.
\end{corollary}

Moreover, Theorem~\ref{thm:trunc} gives a recipe for computing $\schubert_u \stanley_v(x_1,\ldots,x_k)$ in $m-k+\ell$ steps, where $m$ is the smallest index $i$ such that $u_i=i$ for all $i>m$, and $\ell$ is the position of the last descent in $v$. Note that since Stanley symmetric polynomials are stable under $v \mapsto 1^n\times v$, we may always take $v_1>1$ to minimize $\ell$. When $v$ is grassmannian, we can do better still.

\begin{theorem}
  Let $u$ be a permutation, $k$ a positive integer for which $u_i < u_{i+1}$ for any $i>k$, and $\lambda$ a partition of length at most $k$. Let $\ell = \max\{k,i \mid u_i \neq i \}$, and set $m = \ell + \ell(\lambda)$ and $w = u\times_{\ell} v(\lambda,\ell(\lambda))$. Then
  \begin{equation}
    \schubert_u s_{\lambda}(x_1,\ldots,x_k) = \sum_{\substack{a_{i,j} < m-i \\ i=0,\ldots,n < m-k \\ j=0,\ldots,\lambda_1-1}} \schubert_{w D_0 U_0 \cdots D_n U_n}
  \end{equation}
  where $D_i = (m-i,m+\lambda_1-i) \cdots (m-i,m+1-i)$, $U_i = (a_{i,0}, m-i) \cdots (a_{i,\lambda_1-1}, m+(\lambda_1-1)-i)$, and the sum is taken over $\{a_{i,j}\}$ such that $\ell(w D_0 U_0 \cdots D_h U_h) = \ell(w)$ for $h=0,\ldots,n$.
  \label{thm:schub-schur}
\end{theorem}

\begin{proof}
  Let $W$ be a permutation and $k,L,M$ be integers, with $M>k$, such that 
  \[ W_{k+1} < \cdots < W_{M} > W_{M+1} < \cdots < W_{M+L} < W_{M} < W_{M+L+1} < \cdots . \]
  Then $\schubert_W$ is a polynomial in $M$ variables and, by Theorem~\ref{thm:trunc}, we have
  \begin{equation}
    \schubert_{W}(x_1,\ldots,x_{M-1},0) = \sum_{\substack{U = \hat{W} (a_1,M) \cdots (a_1,M+L-1-i) \\ a_j < M \ \mathrm{and} \ \ell(U) = \ell(W)}} \schubert_{U},
  \end{equation}
  where $\hat{W} = W (M,M+L) \cdots (M,M+1)$. Moreover, the structure of $W$ dictates that any $U$ appearing in the summation on the right must have the form
  \[ U_{k+1} < \cdots < U_{M-1} , U_{M} < \cdots < U_{M+L-1} < U_{M-1} < U_{M+L} < \cdots . \]

  By Corollary~\ref{cor:kohnert}, we have $\schubert_u s_{\lambda}(x_1,\ldots,x_k) = \schubert_w(x_1,\ldots,x_k)$.   We show, by induction on $i$, that
  \begin{equation}
    \schubert_{w D_0 U_0 \cdots D_{i-1} U_{i-1}}(x_1,\ldots,x_{m-i-1},0) = \sum_{\substack{w' = \hat{w} (a_1,m-i) \cdots (a_1,m+\lambda_1-1-i) \\ a_j < m-i \ \mathrm{and} \ \ell(u) = \ell(w)}} \schubert_{w'},
    \label{e:ind}
  \end{equation}
  where $\hat{w} = w D_0 U_0 \cdots D_{i-1} U_{i-1} (m-i,m+\lambda_1-i) \cdots (m-i,m+1-i)$. The theorem follows.

  Clearly $w$ has the form of $W$ from above with $L = \lambda_1$ and $M=m$, proving the base case, and also that $w D_0 U_0$ has the form of $U$ with the same $L,M$. Therefore, we may assume, by induction, that $w D_0 U_0 \cdots D_{i-1} U_{i-1}$ has the form of $W$ for $L=\lambda_1$ and $M=m-i$, and conclude by the argument above that \eqref{e:ind} holds and that any term appearing on the right has the form of $U$. 
\end{proof}

For example, taking $u = 42153$, $k=5$, and $\lambda = (2,1)$ gives $w=421537968$, and
\begin{displaymath}
  \schubert_{42153} s_{(3,1,1)}(x_1,\ldots,x_5) = \schubert_{4235716}+\schubert_{4315726}+\schubert_{4216735}+\schubert_{4217536}+\schubert_{5217346}, 
\end{displaymath}
where, for example, we have
\[ 421537968 (7,9)(7,8)(5,7)(6,8)(6,8)(6,7)(3,6)(5,7) = 4235716.\]
Observe the potential cancellation of $(6,8)(6,8)$. Moreover, notice that
\[ w (7,9)(7,8)(5,7)(6,8)(6,8)(6,7)(3,6)(5,7) = w (7,9)(7,8)(6,7)(5,6)(3,6)(5,7) = u (5,6)(3,6)(5,7). \]

In general, applying the canonical reordering of transpositions from the proof of Theorem~\ref{thm:trunc} results in certain sequences of the form $u (a_1,b_1) \cdots (a_n,b_n)$ where $a_i \leq k < b_i$.

\begin{corollary}
  Let $u$ be any permutation, $k$ a positive integer for which $u_i < u_{i+1}$ for any $i>k$, and $\lambda$ any partition of length at most $k$. Then $c^{w}_{u,v(\lambda,k)}$ enumerates a subset of paths $w = u (a_1,b_1) \cdots (a_n,b_n)$ with $a_i \leq k < b_i$ and $\ell(w) = \ell(u)+n$, where $n$ is the size of $\lambda$.
  \label{cor:schub-schur}
\end{corollary}

%
%

\bibliographystyle{amsalpha} 
\bibliography{trunc.bib}

\end{document}